\documentclass[11pt]{amsart}%
\usepackage{amsmath}%
\usepackage{amssymb}%
\usepackage{amscd}%
\usepackage{color}%
\usepackage{graphicx}%
\usepackage{mathrsfs}%
\usepackage[all]{xy}%
\usepackage{newtxtext,newtxmath}%
\usepackage[all, warning]{onlyamsmath}
\usepackage{amsrefs}
\usepackage{bbm}
\def\ol#1{\overline{#1}}
\def\wh#1{\breve{#1}}
\def\wt#1{\widetilde{#1}}
\theoremstyle{plain}
    \newtheorem{theorem}{Theorem}[section]
    \newtheorem{proposition}[theorem]{Proposition}
    \newtheorem{lemma}[theorem]{Lemma}
    \newtheorem{corollary}[theorem]{Corollary}
      
\theoremstyle{definition}
    \newtheorem{definition}[theorem]{Definition}

    \newtheorem{remark}[theorem]{Remark}
\def\Alphabet{A,B,C,D,E,F,G,H,I,J,K,L,M,N,O,P,Q,R,S,T,U,V,W,X,Y,Z}
\def\alphabet{a,b,c,d,e,f,g,h,i,j,k,l,m,n,o,p,q,r,s,t,u,v,w,x,y,z}
\def\endpiece{xxx}
\def\makeAlphabet[#1]{\expandafter\makeA#1,xxx,}
\def\makealphabet[#1]{\expandafter\makea#1,xxx,}
\def\makeA#1,{\def\temp{#1}\ifx\temp\endpiece\else%
\mkbb{#1}\mkfrak{#1}\mkbf{#1}\mkcal{#1}\mkscr{#1}\mkbs{#1}\expandafter\makeA\fi}%
\def\makea#1,{\def\temp{#1}\ifx\temp\endpiece\else\mkfrak{#1}\mkbf{#1}\mkbs{#1}\expandafter\makea\fi}%
\def\mkbb#1{\expandafter\def\csname bb#1\endcsname{\mathbb{#1}}}
\def\mkfrak#1{\expandafter\def\csname fr#1\endcsname{\mathfrak{#1}}}
\def\mkbf#1{\expandafter\def\csname b#1\endcsname{\mathbf{#1}}}
\def\mkcal#1{\expandafter\def\csname c#1\endcsname{\mathcal{#1}}}
\def\mkscr#1{\expandafter\def\csname s#1\endcsname{\mathscr{#1}}}
\def\mkbs#1{\expandafter\def\csname bs#1\endcsname{{\boldsymbol{#1}}}}
\def\makeop[#1]{\xmakeop#1,xxx,}
\def\mkop#1{\expandafter\def\csname #1\endcsname{{\mathrm{#1}}}} %
\def\xmakeop#1,{\def\temp{#1}\ifx\temp\endpiece\else\mkop{#1}\expandafter\xmakeop\fi}%
\def\makeup[#1]{\xmakeup#1,xxx,}
\def\mkup#1{\expandafter\def\csname #1\endcsname{{\mathrm{#1}\,}}} %
\def\xmakeup#1,{\def\temp{#1}\ifx\temp\endpiece\else\mkup{#1}\expandafter\xmakeup\fi}%
\makeAlphabet[\Alphabet]
\makealphabet[\alphabet]
\makeop[Alt,End,Ext,Hom,id,GL,Tr,sgn,ab,Gal,fin,Cl,Hecke,tors,Re,alt,qis]
\makeup[Spec,Im,Ker,Coker,gen,Spf]

\def\bsalpha{{\boldsymbol{\alpha}}}

\begin{document}
\title[Canonical Generating Classes]{Canonical Equivariant
Cohomology Classes Generating Zeta Values of Totally Real Fields}
\author[Bannai]{Kenichi Bannai$^{*\diamond}$}
\email{bannai@math.keio.ac.jp}
\address{${}^*$Department of Mathematics, Faculty of Science and Technology, Keio University, 3-14-1 Hiyoshi, Kouhoku-ku, Yokohama 223-8522, Japan}
\address{${}^\diamond$Mathematical Science Team, RIKEN Center for Advanced Intelligence Project (AIP), 1-4-1 Nihonbashi, Chuo-ku, Tokyo 103-0027, Japan}
\author[Hagihara]{Kei Hagihara$^{\diamond*}$}
\author[Yamada]{Kazuki Yamada$^*$}
\author[Yamamoto]{Shuji Yamamoto$^{*\diamond}$}

\date{\today}

\date{\today \quad (Version 1.08)}
\begin{abstract}
	It is known that the special values at nonpositive integers of a Dirichlet $L$-function
	may be expressed using the generalized Bernoulli numbers, which are defined by a 
	canonical generating function. 
	The purpose of this article is to consider the generalization of
	this classical result to the case of Hecke $L$-functions of totally real fields.
	Hecke $L$-functions may be expressed canonically as a finite
	sum of zeta functions of Lerch type.
	By combining the non-canonical multivariable generating functions 
	constructed by Shintani, 
	we newly construct a canonical class, which we call the \textit{Shintani generating class},
	in the equivariant cohomology of an algebraic torus associated to the totally real field.
	Our main result states that the specializations at torsion points of the derivatives of the 
	Shintani generating class give values at nonpositive integers of the zeta functions of Lerch type.
	This result gives the insight that the correct framework in the higher dimensional case is to consider 
	higher equivariant cohomology classes instead of functions.  
\end{abstract}
\thanks{This research is supported by KAKENHI 18H05233.
The topic of research initiated from the KiPAS program FY2014--2018 of the Faculty of Science and Technology at Keio University.}
\maketitle
\setcounter{tocdepth}{1}

%
%
%
\section{Introduction}\label{section: introduction}
%
%
%

It is classically known that the special values at nonpositive integers of a Dirichlet $L$-function
may be expressed using the generalized Bernoulli numbers, which are defined by a canonical rational generating function. 
This simple but significant result is the basis of the deep connection between the special values of Dirichlet $L$-functions
and important arithmetic invariants pertaining to the abelian extensions of $\bbQ$.
In his ground-breaking article \cite{Shi76}, Shintani generalized this result to the case of Hecke $L$-functions of totally real fields.
His approach consists of two steps:  The decomposition of a Hecke $L$-function into a finite sum
of zeta functions -- the \textit{Shintani zeta functions} -- 
associated to certain cones, and the construction of a multivariable generating function 
for special values of each Shintani zeta function.
Although this method attained certain success, including the construction by Barsky \cite{Bar78} 
and Cassou-Nogu\`es  \cite{CN79} of the 
$p$-adic $L$-functions for totally real fields, the decomposition step above requires
a choice of cones, and the resulting generating function
is non-canonical.  A canonical object behind these generating functions remained to be found.

The purpose of this article is to construct geometrically such a canonical object, which we call the \textit{Shintani generating class},
 through the combination of the following three ideas.
We let $g$ be the degree of the totally real field.
First, the Hecke $L$-functions are expressed canonically
in terms of the \textit{zeta functions of Lerch type} (cf.\ Definition \ref{def: Lerch}),
or simply \textit{Lerch zeta functions},
which are defined for finite additive characters parameterized by torsion points of a 
certain algebraic torus of dimension $g$, originally considered by Katz \cite{Katz81},
associated to the totally real field.
Second, via a \v Cech resolution, the multivariable 
generating functions constructed by Shintani
for various cones  may beautifully be combined 
to form the Shintani generating class, a canonical
cohomology class in the $(g-1)$-st cohomology group of the algebraic torus minus the identity.
Third, the class descends into the equivariant cohomology with respect to the action of totally positive units,
which successfully allows for nontrivial specializations of the class and its derivatives at torsion points.
Our main result, Theorem \ref{theorem: main}, states that the specializations at nontrivial torsion points of the derivatives of the 
Shintani generating class give values at nonpositive integers of the Lerch zeta functions associated to the totally real field.

The classical result for $\bbQ$ that we generalize, viewed through our emphasis on Lerch zeta functions, is as follows.
The Dirichlet $L$-function may canonically be expressed as a finite linear combination of the classical
\textit{Lerch zeta functions}, defined by the series
\begin{equation}\label{eq: Lerch}
	\cL(\xi, s)\coloneqq\sum_{n=1}^\infty \xi(n)n^{-s}
\end{equation}
for finite characters $\xi\in\Hom_\bbZ(\bbZ,\bbC^\times)$.
The series \eqref{eq: Lerch} converges for any $s\in\bbC$ such that $\Re(s)>1$ and has an analytic continuation 
to the whole complex plane, holomorphic if $\xi\neq 1$. 
When $\xi=1$,  the function $\cL(1, s)$ coincides with the Riemann zeta function
$\zeta(s)$, hence has a simple pole at $s=1$.
A crucial property of the Lerch zeta functions is that it has a canonical generating function $\cG(t)$,
which single-handedly captures for \textit{all} nontrivial finite characters $\xi$ the values of Lerch zeta functions at nonpositive integers.

Let $\bbG_m\coloneqq\Spec\bbZ[t,t^{-1}]$ be the multiplicative group, and let $\cG(t)$ be the rational function
\[
	  \cG(t)\coloneqq \frac{t}{1-t}  \in \Gamma\bigl(U,\sO_{\bbG_m}\bigr),
\]
where $U\coloneqq\bbG_m\setminus\{1\}$.
We denote by $\partial$ the algebraic differential operator
$\partial\coloneqq t\frac{d}{dt}$, referred to as the ``magic stick'' in \cite{Kato93}*{1.1.7}.
Note that any $\xi\in\bbG_m(\bbC)$ corresponds to an additive character $\xi\colon\bbZ\rightarrow\bbC^\times$
given by $\xi(n)\coloneqq\xi^n$ for any $n\in\bbZ$.
Then we have the following.

\begin{theorem}\label{theorem: classical generating}
	For any nontrivial torsion point $\xi$ of $\bbG_m$ and $k\in\bbN$, we have
	\[
		\cL(\xi,-k)=\partial^k\cG(t)\big|_{t=\xi}\in \bbQ(\xi).
	\]
	In particular, the values $\cL(\xi,-k)$ for any $k\in\bbN$ are all algebraic.
\end{theorem}

The purpose of this article is to generalize the above result to the case of totally real fields.
Let $F$ be a totally real field of degree $g$, and 
let $\cO_F$ be its ring of integers.
We denote by $\cO_{F+}$ the set of totally positive integers and 
by $\Delta\coloneqq\cO_{F+}^\times$
the set of totally positive units of $F$.  
Let $\bbT\coloneqq\Hom_\bbZ(\cO_F,\bbG_m)$
be an algebraic torus defined over $\bbZ$ which represents the functor 
associating to any $\bbZ$-algebra $R$ the group
$\bbT(R)=\Hom_\bbZ(\cO_F,R^\times)$.
Such a torus was used by Katz \cite{Katz81} to reinterpret the construction by Barsky \cite{Bar78} 
and Cassou-Nogu\`es  \cite{CN79} of the $p$-adic $L$-function of totally real fields.
For the case $F=\bbQ$, we have $\bbT = \Hom_\bbZ(\bbZ,\bbG_m)=\bbG_m$, hence $\bbT$ is a
natural generalization of the multiplicative group.  
For an additive character $\xi\colon\cO_F\rightarrow R^\times$
and $\varepsilon\in\Delta$, we let $\xi^\varepsilon$ be the character defined by 
$\xi^\varepsilon(\alpha)\coloneqq\xi(\varepsilon\alpha)$ for any $\alpha\in\cO_F$.
This gives an action of $\Delta$ on the set of additive characters $\bbT(R)$.

We consider the following zeta function, which we regard as
the generalization of the classical Lerch zeta function to the case of totally real fields.

\begin{definition}\label{def: Lerch}
	For any torsion point $\xi\in\bbT(\bbC)=\Hom_\bbZ(\cO_F,\bbC^\times)$, we define the 
	\textit{zeta function of Lerch type}, or simply the
	\textit{Lerch zeta function}, by
	\begin{equation}\label{eq: Shintani-Lerch}
		\cL(\xi\Delta, s)\coloneqq\sum_{\alpha\in\Delta_{\xi}\backslash\cO_{F+}}\xi(\alpha)N(\alpha)^{-s},
	\end{equation}
	where $N(\alpha)$ is the norm of $\alpha$, and 
	$\Delta_{\xi}\subset\Delta$ is the isotropic subgroup of $\xi$, i.e.\ 
	the subgroup consisting of
	$\varepsilon\in\Delta$ such that $\xi^\varepsilon=\xi$.
\end{definition}

The notation $\cL(\xi\Delta, s)$ is used since \eqref{eq: Shintani-Lerch} depends only on the $\Delta$-orbit of $\xi$.
This series is known to converge for $\Re(s)>1$, and may be continued analytically
to the whole complex plane.
When the narrow class number of $F$ is \textit{one}, 
the Hecke $L$-function of a finite Hecke character  of $F$
may canonically be expressed as a finite linear sum of $\cL(\xi\Delta, s)$ for suitable finite  characters $\xi$
(see Proposition \ref{prop: Hecke}).  

The action of $\Delta$ on additive characters 
gives a right action of $\Delta$ on $\bbT$.
The structure sheaf $\sO_\bbT$ on $\bbT$ has a natural $\Delta$-equivariant structure in the sense of Definition 
\ref{def: equivariant structure}.
Let $U\coloneqq\bbT\setminus\{1\}$. 
Our main results are as follows.

\begin{theorem}\label{theorem: introduction}
	\begin{enumerate}
		\item \mbox{(Proposition \ref{prop: Shintani generating class}) }
		There exists a canonical class
		\[
			\cG\in H^{g-1}(U/\Delta,\sO_\bbT),
		\]
		where $H^{g-1}(U/\Delta,\sO_\bbT)$ is the equivariant 
		cohomology of $U$ with coefficients in $\sO_\bbT$ 
		(see \S\ref{section: equivariant} for the precise definition.)
		\item (Theorem \ref{theorem: main}) 
		For any nontrivial torsion point $\xi$ of $\bbT$, we have a canonical
		isomorphism
		\[
			H^{g-1}(\xi/\Delta_\xi,\sO_\xi)\cong\bbQ(\xi).
		\]
		Through this isomorphism, for any integer $k\geq 0$,
		we have
		 \[
		 	\cL(\xi\Delta,-k)=\partial^k\cG(\xi)\in\bbQ(\xi),
		\]		
		where
		$\partial\colon H^{g-1}(U/\Delta,\sO_\bbT)\rightarrow H^{g-1}(U/\Delta,\sO_\bbT)$
		is a certain differential operator given in \eqref{eq: differential}, and 
		$\partial^k\cG(\xi)$ is the image of $\partial^k\cG$ with respect
		to the specialization map
		$
			H^{g-1}(U/\Delta,\sO_\bbT)\rightarrow H^{g-1}(\xi/\Delta_\xi,\sO_\xi)
		$
		induced by the  equivariant morphism $\xi\rightarrow U$.	
		\end{enumerate}
\end{theorem}

We refer to the class $\cG$ as the \textit{Shintani generating class}.
If $F=\bbQ$, then we have $\Delta=\{1\}$, and the class $\cG$ is simply the rational function 
$\cG(t)=t/(1-t)\in H^0(U,\sO_{\bbG_m})=\Gamma(U,\sO_{\bbG_m})$.  Thus Theorem \ref{theorem: introduction} (2) 
coincides with Theorem \ref{theorem: classical generating} in this case.
For the case $F=\bbQ$ and also for the case of imaginary quadratic fields (see for example \cite{CW77}\cite{CW78}), 
canonical algebraic generating functions
of special values of Hecke $L$-functions play a crucial
role in relating the special values of Hecke $L$-functions
to arithmetic invariants.
However, up until now, the discovery of such a \textit{canonical} 
generating function has been elusive in the higher dimensional cases.
Our result suggests that the correct framework in the higher dimensional case is to consider 
equivariant cohomology classes instead of functions.  

Relation of our work to the results of
Charollois, Dasgupta, and Greenberg 
\cite{CDG14} was kindly pointed out to us by Peter Xu.
As a related result, the relation of special values of Hecke $L$-functions of totally 
real fields to the topological polylogarithm 
on a torus was studied by Be\u\i linson, Kings, and Levin in \cite{BKL18}.
The polylogarithm for general commutative group schemes were constructed by Huber and Kings \cite{HK18}.  Our discovery of the Shintani generating class arose from our attempt to explicitly describe 
various realizations of the polylogarithm for the algebraic torus $\bbT$.
In subsequent research, we will explore the arithmetic implications of our insight
(see for example \cite{BHY00}).

\tableofcontents

The content of this article is as follows.  In \S \ref{section: Lerch Zeta}, we will introduce the
Lerch zeta function $\cL(\xi\Delta, s)$ and show that
this function may be expressed non-canonically as a
linear sum of Shintani zeta functions. 
We will then review the multivariable generating function 
constructed by Shintani of the special values of Shintani zeta functions.
In \S \ref{section: equivariant}, we will define the equivariant cohomology of a scheme
with an action of a group, and will construct the equivariant \v Cech complex $C^\bullet(\frU/\Delta,\sF)$ which calculates
the equivariant cohomology of $U\coloneqq\bbT\setminus\{1\}$ with coefficients in an equivariant coherent sheaf $\sF$
on $U$.  In \S \ref{section: Shintani Class}, we will define in Proposition \ref{prop: Shintani generating class}
the Shintani generating class $\cG$, and in Lemma \ref{lem: differential} give the definition of the derivatives.
Finally in \S \ref{section: specialization}, we will give the proof of our main theorem, Theorem \ref{theorem: main}, 
which coincides with Theorem \ref{theorem: introduction} (2).

%
%
%
\section{Lerch Zeta Function}\label{section: Lerch Zeta}
%
%
%

In this section, we first introduce the Lerch zeta function for totally real fields.
Then we will then define the Shintani zeta function associated to a cone $\sigma$ and a function $\phi\colon\cO_F\rightarrow\bbC$
which factors through $\cO_F/\frf$ for some nonzero ideal $\frf\subset\cO_F$.
We will then describe the generating function of its values at nonpositive integers when $\phi$ is a finite additive character.

Let $\xi\in\Hom_\bbZ(\cO_F,\bbC^\times)$ be a $\bbC$-valued character on $\cO_F$ of finite order.
As in Definition \ref{def: Lerch} of \S\ref{section: introduction},
we define the Lerch zeta function for totally real fields by the series
\[
	\cL(\xi\Delta, s)\coloneqq\sum_{\alpha\in\Delta_\xi\backslash\cO_{F+}}\xi(\alpha)N(\alpha)^{-s},
\]
where $\Delta_\xi\coloneqq\{\varepsilon\in\Delta\mid \xi^\varepsilon=\xi\}$,
which may be continued analytically to the whole complex plane.

\begin{remark}
	Note that we have
	\[
		\cL(\xi\Delta, s)=\sum_{\alpha\in\Delta\backslash\cO_{F+}}
		\sum_{\varepsilon\in\Delta_\xi\backslash\Delta} \xi(\varepsilon\alpha)N(\alpha)^{-s}.
	\]
	Even though $\xi(\alpha)$ is not well-defined for $\alpha\in\Delta\backslash\cO_F$, 
	the sum $\sum_{\varepsilon\in\Delta_\xi\backslash\Delta}\xi(\varepsilon\alpha)$
	is well-defined for $\alpha\in\Delta\backslash\cO_F$.
\end{remark}

The importance of  $\cL(\xi\Delta, s)$ is in its relation to the Hecke $L$-functions of $F$.
Let $\frf$ be a nonzero integral ideal of $F$. 
We denote by $\Cl^+_F(\frf)\coloneqq I_\frf/P^+_\frf$
the strict ray class group modulo $\frf$ of $F$, where
$I_\frf$ is the group of fractional ideals of $F$ prime to $\frf$
and $P^+_\frf\coloneqq \{  (\alpha)  \mid \alpha\in F_+, \alpha\equiv 1 \operatorname{mod}^\times \frf\}$.
A finite Hecke character of $F$ of conductor $\frf$ is a character
\[
	\chi\colon\Cl^+_F(\frf)\rightarrow\bbC^\times.
\]
By \cite{Neu99}*{Chapter VII (6.9) Proposition}, there exists a unique character
$
	\chi_\fin\colon(\cO_F/\frf)^\times\rightarrow\bbC^\times
$
associated to $\chi$ such that $\chi((\alpha))=\chi_\fin(\alpha)$ for any $\alpha\in\cO_{F+}$ prime to $\frf$.  
In particular, we have $\chi_\fin(\varepsilon)=1$ for any $\varepsilon\in\Delta$.
Extending by \textit{zero}, we regard $\chi_\fin$ as functions on $\cO_F/\frf$ and $\cO_F$ with values in $\bbC$.

In what follows, we let $\bbT[\frf]\coloneqq\Hom(\cO_F/\frf,{\ol\bbQ}^\times)\subset\bbT(\ol\bbQ)$
be the set of $\frf$-torsion points of $\bbT$.
We say that a character $\chi$, $\chi_\fin$ or $\xi\in\bbT[\frf]$ is \textit{primitive},
if it does not factor respectively through $\Cl_F^+(\frf')$, $(\cO_F/\frf')^\times$ or $\cO_F/\frf'$
for any integral ideal $\frf'\neq\frf$ such that $\frf'|\frf$. 
Then we have the following.

\begin{lemma}\label{lemma: fourier}
	For any $\xi\in\bbT[\frf]$, let
	\[
		c_\chi(\xi)\coloneqq\frac{1}{\bsN(\frf)}\sum_{\beta\in\cO_F/\frf} \chi_\fin(\beta)\xi(-\beta).
	\]
	Then we have
	\[
		\chi_\fin(\alpha)=\sum_{\xi\in\bbT[\frf]}c_\chi(\xi)\xi(\alpha).
	\]
	Moreover, if $\chi_\fin$ is primitive, then we have $c_\chi(\xi)=0$ for any non-primitive $\xi$.
\end{lemma}

\begin{proof}
	The first statement follows from
	\[
		\sum_{\xi\in\bbT[\frf]}c_\chi(\xi)\xi(\alpha) = \frac{1}{\bsN(\frf)} \sum_{\beta\in\cO_F/\frf}  
		\chi_\fin(\beta)\biggl(\sum_{\xi\in\bbT[\frf]} \xi(\alpha-\beta)\biggr) = \chi_\fin(\alpha),
	\]
	where the last equality follows from the fact that $\sum_{\xi\in\bbT[\frf]} \xi(\alpha)=N(\frf)$ if $\alpha\equiv 0\pmod{\frf}$ and 
	$\sum_{\xi\in\bbT[\frf]} \xi(\alpha)=0$ if $\alpha\not\equiv 0\pmod{\frf}$.
	Next, suppose $\chi_\fin$ is primitive, and let $\frf'\neq\frf$ be an integral ideal of $F$ such that $\frf'|\frf$ and $\xi\in\bbT[\frf ']$. 
	Since $\chi_\fin$ is primitive, it does not factor through $\cO_F/\frf'$, hence there exists an element $\gamma\in\cO_F$ 
	prime to $\frf$
	such that $\gamma\equiv1\pmod{\frf'}$ and $\chi_\fin(\gamma)\neq 1$.  Then since $\xi\in\bbT[\frf']$, we have
	$\xi(\gamma\alpha)=\xi(\alpha)$ for any $\alpha\in\cO_F$.  This gives
	\begin{align*}
		c_\chi(\xi)=\frac{1}{\bsN(\frf)}\sum_{\beta\in\cO_F/\frf} \chi_\fin(\beta)\xi(-\beta)
		&=\frac{1}{\bsN(\frf)}\sum_{\beta\in\cO_F/\frf} \chi_\fin(\beta)\xi(-\gamma\beta)\\
		&=\frac{\ol\chi_\fin(\gamma)}{\bsN(\frf)}\sum_{\beta\in\cO_F/\frf} \chi_\fin(\gamma\beta)
		\xi(-\gamma\beta)=\ol\chi_\fin(\gamma)c_\chi(\xi).
	\end{align*}
	Since $\chi_\fin(\gamma)\neq 1$, we have $c_\chi(\xi)=0$ as desired.
\end{proof}

Note that since multiplication by $\varepsilon\in\Delta$ is bijective on $\cO_F/\frf$ and since $\chi_\fin(\varepsilon)=1$,
we have $c_\chi(\xi^\varepsilon)=c_\chi(\xi)$. Then we have the following.

\begin{proposition}\label{prop: Hecke}
	Assume that the narrow class number of $F$ is \textit{one},
	and let $\chi\colon\Cl^+_F(\frf)\rightarrow\bbC^\times$ be a finite primitive Hecke character of $F$ of 
	conductor $\frf\neq(1)$.
	Then for $U[\frf]\coloneqq\bbT[\frf]\setminus\{1\}$, we have
	\[
		L(\chi, s)=\sum_{\xi\in U[\frf]/\Delta} c_\chi(\xi)\cL(\xi\Delta, s).
	\]
\end{proposition}
\begin{proof}
	By definition and Lemma \ref{lemma: fourier}, we have
	\begin{align*}
		\sum_{\xi\in\bbT[\frf]/\Delta} c_\chi(\xi)\cL(\xi\Delta, s)
		&=\sum_{\xi\in\bbT[\frf]/\Delta}\sum_{\alpha\in\Delta\backslash\cO_{F+}}
		\sum_{\varepsilon\in\Delta_\xi\backslash\Delta}c_\chi(\xi)\xi(\varepsilon\alpha)\bsN(\alpha)^{-s}\\
		&=\sum_{\alpha\in\Delta\backslash\cO_{F+}}\sum_{\xi\in\bbT[\frf]/\Delta}\sum_{\varepsilon\in\Delta_\xi\backslash\Delta}
		c_\chi(\xi^\varepsilon)\xi^\varepsilon(\alpha)\bsN(\alpha)^{-s}\\
		&=\sum_{\alpha\in\Delta\backslash\cO_{F+}}\sum_{\xi\in\bbT[\frf]}c_\chi(\xi)\xi(\alpha)\bsN(\alpha)^{-s}\\
		&=\sum_{\alpha\in\Delta\backslash\cO_{F+}}
		\chi_\fin(\alpha)\bsN(\alpha)^{-s} = \sum_{\fra\subset\cO_F} \chi(\fra)\bsN\!\fra^{-s}.
	\end{align*}
	Our assertion follows from the definition of the Hecke $L$-function and the fact that $c_\chi(\xi)=0$ for $\xi=1$.
\end{proof}

\begin{remark}\label{rem: class number one}
	We assumed the condition on the narrow class number for simplicity.
	By considering the Lerch zeta functions corresponding to additive characters in 
	$\Hom_\bbZ(\fra,\bbC^\times)$ for general fractional ideals $\fra$ of $F$,
	we may express the Hecke $L$-functions when the narrow class number of 
	$F$ is greater than \textit{one}.
\end{remark}

We will next define the Shintani zeta function 
associated to a cone.
Note that we have a canonical isomorphism
\[
	F\otimes\bbR\cong\bbR^I\coloneqq \prod_{\tau\in I}\bbR,  \qquad \alpha\otimes 1 \mapsto (\alpha^\tau),
\]
where $I$ is the set of embeddings $\tau\colon F\hookrightarrow\bbR$ 
and we let $\alpha^\tau\coloneqq\tau(\alpha)$ for any embedding $\tau\in I$.
We denote by $\bbR^I_+\coloneqq\prod_{\tau\in I}\bbR_+$
the set of totally positive elements of $\bbR^I$, where $\bbR_+$ is the set of positive real numbers.

\begin{definition}
	A rational closed polyhedral cone in $\bbR^I_+\cup\{0\}$, which we simply call a cone, is any set of the form
	\[
		\sigma_\bsalpha\coloneqq\{  x_1 \alpha_1+\cdots+x_m\alpha_m \mid x_1,\ldots,x_m \in\bbR_{\geq0}\}
	\]
	for some $\bsalpha=(\alpha_1,\ldots,\alpha_m)\in\cO_{F+}^m$.  
	In this case, we say that  $\bsalpha$ is a generator of $\sigma_\bsalpha$.
	By considering the case $m=0$,  we see that $\sigma=\{0\}$ is a cone.
\end{definition}

We define the dimension  $\dim\sigma$  of a cone $\sigma$ to be the dimension 
of the $\bbR$-vector space generated by $\sigma$.
In what follows, we fix a numbering $I=\{\tau_1,\ldots,\tau_g\}$ of elements in $I$.
For any subset $R\subset\bbR_+^I$, we let
\[
	\wh R\coloneqq\{(u_{\tau_1},\ldots,u_{\tau_g})\in\bbR_+^I \mid\exists\delta>0,\,0<\forall\delta'<\delta,(u_{\tau_1},\ldots,u_{\tau_{g-1}},
	u_{\tau_g}-\delta')\in R\}.
\]

\begin{definition}
	Let $\sigma$ be a cone, and let $\phi\colon\cO_F\rightarrow\bbC$ be a $\bbC$-valued function on $\cO_F$
	which factors through $\cO_F/\frf$ for some nonzero ideal $\frf\subset\cO_F$.
	We define the \textit{Shintani zeta function} $\zeta_\sigma(\phi,\bss)$ associated to a cone $\sigma$ and
	function $\phi$ by the series
	\begin{equation}\label{eq: Shintani zeta}
		\zeta_\sigma(\phi,\bss)\coloneqq\sum_{\alpha\in\wh\sigma\cap\cO_F} \phi(\alpha) \alpha^{-\bss},
	\end{equation}
	where $\bss=(s_{\tau})\in\bbC^I$ and $\alpha^{-\bss}\coloneqq\prod_{\tau\in I}(\alpha^{\tau})^{-s_{\tau}}$.
	The series \eqref{eq: Shintani zeta} converges if $\Re(s_{\tau})>1$ for any $\tau\in I$.
\end{definition}

By \cite{Shi76}*{Proposition 1}, the function $\zeta_\sigma(\phi,\bss)$ has a meromorphic continuation to any $\bss\in\bbC^I$.
If we let $\bss=(s,\ldots, s)$ for $s\in\bbC$, then we have
\begin{equation}\label{eq: diagonal}
	\zeta_\sigma(\phi,(s,\ldots, s))=\sum_{\alpha\in\wh\sigma\cap\cO_F} \phi(\alpha) N(\alpha)^{-s}.
\end{equation}

Shintani constructed the generating function of values of $\zeta_\sigma(\xi, \bss)$ 
at nonpositive integers for additive characters $\xi\colon\cO_F\rightarrow\bbC^\times$ of finite order,
given as follows.
In what follows, we view $z\in F\otimes\bbC$ as an element 
$z = (z_\tau)\in\bbC^I$ through the canonical isomorphism $F\otimes\bbC\cong\bbC^I$.
%

\begin{definition}\label{def: generating}
	Let $\sigma=\sigma_\bsalpha$ be a $g$-dimensional cone generated by 
	$\bsalpha=(\alpha_1,\ldots,\alpha_g)\in\cO_{F+}^g$,
	and we let $P_\bsalpha\coloneqq\{ x_1\alpha_1+\cdots+x_g\alpha_g\mid
	\forall i\,\,0\leq x_i < 1\}$
	be the parallelepiped spanned by $\alpha_1,\ldots,\alpha_g$.
	We define $\sG_{\sigma}(z)$ to be the meromorphic function 
	on $F\otimes\bbC\cong\bbC^I$ given by
	\[
		\sG_\sigma(z)\coloneqq
		\frac{\sum_{\alpha\in \wh P_\bsalpha\cap\cO_F}e^{2\pi i\Tr(\alpha z)}}{{\bigl(1-e^{2\pi i\Tr(\alpha_1z)}}\bigr)
		\cdots\bigl(1-e^{2\pi i\Tr(\alpha_gz)}\bigr)},
	\]
	where $\Tr(\alpha z)\coloneqq\sum_{\tau\in I}\alpha^\tau z_\tau$ for any $\alpha\in\cO_F$.
	The definition of $\sG_\sigma(z)$ 
	depends only on the cone and is independent of the choice of the generator $\bsalpha$.
\end{definition}

\begin{remark}
	If $F=\bbQ$ and $\sigma=\bbR_{\geq0}$, then we have
	$
		\sG_\sigma(z) = \frac{e^{2\pi iz}}{1-e^{2\pi i z}}.
	$
\end{remark}

For $\bsk=(k_\tau)\in\bbN^I$,  we denote $\partial^\bsk\coloneqq\prod_{\tau\in I}\partial_\tau ^{k_\tau}$, 
where $\partial_\tau :=\frac{1}{2\pi i}\frac{\partial}{\partial z_\tau}$. 
For $u\in F$, 
we let $\xi_u$ be the finite additive character on $\cO_F$ defined by $\xi_u(\alpha)\coloneqq e^{2\pi i\Tr(\alpha u)}$.
We note that any additive character on $\cO_F$ with values in $\bbC^\times$ of finite order
is of this form for some $u\in F$.
The following theorem, based on the work of Shintani, is standard (see for example \cite{CN79}*{Th\'eor\`eme 5}, \cite{Col88}*{Lemme 3.2}).

\begin{theorem}\label{theorem: Shintani} 
	Let $\bsalpha$ and $\sigma$ be as in Definition \ref{def: generating}.
	For any $u\in F$ satisfying $\xi_u(\alpha_j)\neq1$ for $j=1,\ldots,g$,
	we have
	\[
		\partial^\bsk \sG_\sigma(z)\big|_{z=u\otimes1}=\zeta_\sigma(\xi_u,-\bsk).
	\]
\end{theorem}
Note that the condition  $\xi_u(\alpha_j)\neq1$ for $j=1,\ldots,g$ ensures that  $z=u\otimes 1$
does not lie on the poles of 
the function $\sG_\sigma(z)$.

The Lerch zeta function $\cL(\xi\Delta, s)$ may be expressed as a finite sum of 
functions $\zeta_\sigma(\xi,(s,\ldots, s))$ 
using the Shintani decomposition.  We first review the definition of the Shintani decomposition.
We say that a cone $\sigma$ is \textit{simplicial}, if there exists a generator of $\sigma$ 
that is linearly independent over $\bbR$.
Any cone generated by a subset of such a generator is called a \textit{face} of $\sigma$.
A simplicial fan $\Phi$ is a set of simplicial cones such that 
for any $\sigma\in\Phi$, any face of $\sigma$ 
is also in $\Phi$, and for any cones $\sigma,\sigma'\in\Phi$, the intersection $\sigma\cap\sigma'$ is 
a common face of $\sigma$ and $\sigma'$.

A version of Shintani decomposition that we will use in this article is as follows.

\begin{definition}\label{def: Shintani}
	A \textit{Shintani decomposition} is a simplicial fan $\Phi$
	satisfying the following properties.
	\begin{enumerate}
		\item $\bbR_+^I\cup\{0\}=\coprod_{\sigma\in\Phi}\sigma^\circ$, where $\sigma^\circ$ is the relative interior of $\sigma$,
		i.e., the interior of $\sigma$ in the $\bbR$-linear span of $\sigma$.
		\item For any $\sigma\in\Phi$ and $\varepsilon\in\Delta$, we have $\varepsilon\sigma\in\Phi$.
		\item The quotient $\Delta\backslash\Phi$ is a finite set.
	\end{enumerate}
\end{definition}

We may obtain such decomposition
by slightly modifying the construction of Shintani \cite{Shi76}*{Theorem 1}
(see also  \cite{Hid93}*{\S2.7 Theorem 1}, \cite{Yam10}*{Theorem 4.1}).
Another construction was given by Ishida \cite{Ish92}*{p.84}.
For any integer $q\geq0$,
we denote by $\Phi_{q+1}$ the subset of $\Phi$ consisting of cones of dimension $q+1$.
Note that by \cite{Yam10}*{Proposition 5.6}, $\Phi_g$ satisfies 
\begin{equation}\label{eq: upper closure}
	\bbR_+^I=\coprod_{\sigma\in\Phi_g}\wh\sigma.
\end{equation}
This gives the following result.

\begin{proposition}
	Let $\xi\colon\cO_F\rightarrow\bbC^\times$ be a  character of finite order, and $\Delta_\xi\subset\Delta$ 
	its isotropic subgroup.
	If $\Phi$ is a Shintani decomposition, then we have
	\begin{equation}\label{eq: Shintani}
		\cL(\xi\Delta, s) = \sum_{\sigma\in\Delta_\xi\backslash\Phi_g}\zeta_\sigma(\xi,(s,\ldots,s)).
	\end{equation}
\end{proposition}

\begin{proof}
	By  \eqref{eq: upper closure}, if $C$ is a representative of $\Delta_\xi\backslash\Phi_g$,
	then $\coprod_{\sigma\in C}\wh\sigma$ is a representative of the set $\Delta_\xi\backslash\bbR_+^I$.
	Our result follows from the definition of the Lerch zeta function and \eqref{eq: diagonal}.
\end{proof}

The expression \eqref{eq: Shintani} is non-canonical, since it depends on the choice of the Shintani decomposition.

%
%
%
\section{Equivariant Coherent Cohomology}\label{section: equivariant}
%
%
%

In this section, we will first give the definition of equivariant sheaves and equivariant cohomology
of a scheme with an action of a group.
As in \S\ref{section: introduction}, we let
\begin{equation}\label{eq: algebraic torus}
	\bbT\coloneqq \Hom_\bbZ(\cO_F,\bbG_m)
\end{equation}
be the algebraic torus over $\bbZ$ defined by Katz \cite[\S 1]{Katz81},
satisfying $\bbT(R)=\Hom_\bbZ(\cO_F,R^\times)$ for any $\bbZ$-algebra $R$.
We will then construct the \textit{equivariant \v Cech complex},
which is an explicit complex which may be used to describe equivariant cohomology of $U\coloneqq\bbT\setminus\{1\}$
with action of $\Delta$.

\begin{remark}
	In order to consider the values of Hecke $L$-functions when the narrow class number of $F$ is greater than
	\textit{one}  (cf. Remark \ref{rem: class number one}), then it would be necessary to consider the algebraic tori
	\[
		\bbT_\fra\coloneqq \Hom_\bbZ(\fra,\bbG_m)
	\]
	for general fractional ideals $\fra$ of $F$.
\end{remark}

We first review the basic facts concerning sheaves on schemes that are equivariant with respect to an action of a group.
Let $G$ be a group with identity $e$.  A $G$-scheme is a scheme $X$ equipped with 
a right action of $G$.  We denote by 
 $[u]\colon X \rightarrow X$ the action of $u\in G$, so that $[uv] = [v]\circ[u]$ for any $u,v\in G$ holds.
In what follows, we let $X$ be a $G$-scheme.

\begin{definition}\label{def: equivariant structure}
	A \textit{$G$-equivariant structure} on
	an $\sO_X$-module $\sF$ is a family of isomorphisms
	\[
		\iota_u\colon [u]^* \sF \xrightarrow\cong \sF
	\]
	for $u\in G$, such that $\iota_e = \id_\sF$ and the diagram
	\[\xymatrix{
		 [uv]^* \sF \ar[r]^-{\iota_{uv}}\ar@{=}[d] &  \sF  \\
		 [u]^*[v]^* \sF  \ar[r]^-{[u]^*\iota_v} &   [u]^*\sF \ar[u]_{\iota_u}
	}\]
	is commutative.  We call $\sF$ equipped with a $G$-equivariant structure a \textit{$G$-equivariant sheaf}.
\end{definition}

Note that the structure sheaf $\sO_X$ itself is naturally a $G$-equivariant sheaf.
For any $G$-equivariant sheaf $\sF$ on $X$,
we define the equivariant global section by $\Gamma(X/G,\sF)\coloneqq
\Hom_{\bbZ[G]}(\bbZ,\Gamma(X, \sF)) = \Gamma(X,\sF)^G$. 
Then the equivariant cohomology 
$H^m(X/G, -)$ is defined to be the
$m$-th right derived functor of $\Gamma(X/G,-)$.

Suppose we have a group homomorphism $\pi\colon G\rightarrow H$.
For a $G$-scheme $X$ and an $H$-scheme $Y$, we say that a morphism $f\colon X\rightarrow Y$
of schemes is \textit{equivariant} with respect to $\pi$, if we have $f\circ[u]=[\pi(u)]\circ f$ for any $u\in G$.
If $\sF$ is a $H$-equivariant sheaf on $Y$ and $f$ is equivariant, then $f^*\sF$ is naturally an $G$-equivariant sheaf on $X$
with the equivariant structure given by $f^*\iota_{\pi(u)}\colon [u]^* (f^*\sF) = f^* ([\pi(u)]^* \sF) \rightarrow f^*\sF$ 
for any $u\in G$, and $f$ induces the pull-back homomorphism
\begin{equation}\label{eq: pullback}
	f^*\colon H^m(Y/H,\sF) \rightarrow H^m(X/G, f^* \sF)
\end{equation}
on equivariant cohomology.

\medskip

We now consider our case of the algebraic torus $\bbT$.
For any $\alpha \in \cO_F$, the morphism $\bbT(R) \rightarrow R^\times$ defined
by mapping $\xi \in \bbT(R)$ to $\xi(\alpha) \in R^\times$
induces a morphism of group schemes $t^\alpha\colon\bbT \rightarrow \bbG_m$, which gives a rational function of $\bbT$.
Then we have
\[
	\bbT=\Spec\bbZ[t^\alpha\mid\alpha\in\cO_F],
\]
where $t^\alpha, t^{\alpha'}$ satisfies the relation $t^\alpha t^{\alpha'}=t^{\alpha+\alpha'}$ for any $\alpha,\alpha'\in\cO_F$.
If we take a basis $\alpha_1,\ldots,\alpha_g$ of $\cO_F$ as a $\bbZ$-module, then we have
\[
	\Spec\bbZ[t^\alpha\mid\alpha\in\cO_F]=\Spec\bbZ[ t^{\pm\alpha_1},\ldots,t^{\pm \alpha_g}]\cong\bbG_m^g.
\]
The action of $\Delta$ on $\cO_F$ by multiplication induces an action of $\Delta$ on $\bbT$.
Explicitly, the isomorphism $[\varepsilon]\colon\bbT\rightarrow\bbT$
for $\varepsilon\in\Delta$ is given by  $t^\alpha\mapsto t^{\varepsilon\alpha}$
for any $\alpha\in \cO_F$.  

\begin{definition}\label{def: twist}
	For any $\bsk=(k_{\tau})\in\bbZ^I$, we define a $\Delta$-equivariant sheaf $\sO_\bbT(\bsk)$ on $\bbT$
	as follows.
	As an $\sO_\bbT$-module we let
	$\sO_\bbT(\bsk)\coloneqq\sO_\bbT$. The $\Delta$-equivariant structure
	\[
		\iota_\varepsilon\colon[\varepsilon]^*\sO_\bbT\cong\sO_\bbT
	\]
	is given by multiplication by 
	$\varepsilon^{-\bsk}\coloneqq\prod_{\tau\in I}(\varepsilon^{\tau})^{-k_{\tau}}$ for any 
	$\varepsilon\in\Delta$.
	Note that for $\bsk, \bsk'\in\bbZ^I$, 
	we have $\sO_\bbT(\bsk)\otimes\sO_\bbT(\bsk')=\sO_\bbT(\bsk+\bsk')$.	
	For the case $\bsk=(k,\ldots,k)$, we have  $\varepsilon^{-\bsk}=N(\varepsilon)^{-k}=1$ for any $\varepsilon\in\Delta$, hence 
	$\sO_\bbT(\bsk)=\sO_\bbT$.
\end{definition}

The open subscheme $U\coloneqq\bbT\setminus\{1\}$ also carries a natural $\Delta$-scheme structure.
We will now construct the \textit{equivariant \v Cech complex},
which may be used to express the cohomology of
$U$ with coefficients in a $\Delta$-equivariant quasi-coherent $\sO_U$-module $\sF$.
For any $\alpha\in\cO_{F}$, we let $U_\alpha\coloneqq\bbT\setminus \{t^{\alpha}=1\}$.  
Then any $\varepsilon\in\Delta$ induces an isomorphism $[\varepsilon]\colon U_{\varepsilon\alpha}\rightarrow U_{\alpha}$.
We say that $\alpha\in\cO_{F+}$ is \textit{primitive} if $\alpha/N\not\in\cO_{F+}$ for any integer $N>1$.
In what follows, we let $A\subset\cO_{F+}$ be the set of primitive elements of $\cO_{F+}$.
Then 
\begin{enumerate}
	\item $\varepsilon A = A$ for any $\varepsilon\in\Delta$.
	\item The set $\frU\coloneqq\{U_\alpha\}_{\alpha\in A}$ gives an affine open covering of $U$.
\end{enumerate}
We note that for any simplicial cone $\sigma$ of dimension $m$, there exists a  generator $\bsalpha\in A^m$,
unique up to permutation of the components.

Let $q$ be an integer $\geq 0$.
For any $\bsalpha=(\alpha_0,\ldots,\alpha_{q})\in A^{q+1}$, 
we let $U_\bsalpha\coloneqq U_{\alpha_0}\cap\cdots\cap U_{\alpha_{q}}$, 
and we denote by $j_{\bsalpha}\colon U_\bsalpha\hookrightarrow U$
the inclusion.  We let
\[
	\sC^q(\frU,\sF)\coloneqq \prod_{\bsalpha\in A^{q+1}}^\alt  j_{\bsalpha*}j_\bsalpha^*\sF
\]
be the subsheaf of 
$ \prod_{\bsalpha\in A^{q+1}}  j_{\bsalpha*}j_\bsalpha^*\sF$ consisting of sections 
$\bss=(s_\bsalpha)$ such that $s_{\rho(\bsalpha)} = \sgn(\rho)s_\bsalpha$ for any $\rho\in\frS_{q+1}$
and $s_\bsalpha=0$ if $\alpha_i=\alpha_j$ for some $i\neq j$.
We define the differential
$
	d^q\colon\sC^q(\frU,\sF)\rightarrow\sC^{q+1}(\frU,\sF)
$
to be the usual alternating sum
\begin{equation}\label{eq: Cech differential}
	(d^qf)_{\alpha_0\cdots\alpha_{q+1}}\coloneqq\sum_{j=0}^{q+1}(-1)^j
	 f_{\alpha_0\cdots\check\alpha_{j}\cdots\alpha_{q+1}}\big|_{U_{(\alpha_0,\ldots,\alpha_{q+1})}\cap V}
\end{equation}
for any section $(f_\bsalpha)$ of $\sC^q(\frU,\sF)$ of each open set $V\subset U$.  
If we let $\sF\hookrightarrow \sC^0(\frU,\sF)$ be the natural inclusion, then we have the exact sequence
\[\xymatrix{
	0\ar[r]& \sF\ar[r]& \sC^0(\frU,\sF)\ar[r]^{d^0}&\sC^1(\frU,\sF)\ar[r]^{\quad d^1}&\cdots  \ar[r]^{d^{q-1}\quad}&\sC^q(\frU,\sF)\ar[r]^{\quad d^q}&\cdots.
}\]

We next consider the action of $\Delta$.
For any $\bsalpha\in A^{q+1}$ and $\varepsilon\in\Delta$, we have a commutative diagram
\[
	\xymatrix{%
		U_{\varepsilon\bsalpha}\ar@{^{(}->}[r]^{j_{\varepsilon\bsalpha}}\ar[d]_{[\varepsilon]}^\cong&U\ar[d]^\cong_{[\varepsilon]}\\
		U_\bsalpha\ar@{^{(}->}[r]^{j_\bsalpha} & U,
	}%
\]
where $\varepsilon\bsalpha\coloneqq(\varepsilon\alpha_0,\ldots,\varepsilon\alpha_{q})$. 
Then we have an isomorphism
\[
	 [\varepsilon]^* j_{\bsalpha*}j_\bsalpha^*\sF
	 \cong j_{\varepsilon\bsalpha*}j_{\varepsilon\bsalpha}^*[\varepsilon]^*\sF
	 \xrightarrow\cong  j_{\varepsilon\bsalpha*}j_{\varepsilon\bsalpha}^*\sF,
\]
where the last isomorphism is induced by 
the $\Delta$-equivariant structure $\iota_\varepsilon\colon[\varepsilon]^*\sF\cong\sF$.
This induces an isomorphism
$
	\iota_\varepsilon\colon[\varepsilon]^*\sC^q(\frU,\sF)\xrightarrow\cong\sC^q(\frU,\sF),
$
which is compatible with the differential \eqref{eq: Cech differential}.  Hence 
$\sC^\bullet(A,\sF)$ is a complex of $\Delta$-equivariant sheaves on $U$.

\begin{proposition}\label{proposition: acyclic}
	The sheaf $\sC^q(\frU,\sF)$ is acyclic with respect to the functor $\Gamma(U/\Delta, -)$.
\end{proposition}

\begin{proof}
	By definition, the functor $\Gamma(U/\Delta, -)$ is the composite of the functors  
	$\Gamma(U,-)$ and $\Hom_{\bbZ[\Delta]}(\bbZ,-)$.  	
	Standard facts concerning the composition of functors shows that we have a spectral sequence
	\[
		E_2^{a,b}=H^a\bigl(\Delta, H^b(U, \sC^q(\frU,\sF))\bigr)\Rightarrow H^{a+b}(U/\Delta, \sC^q(\frU,\sF)).
	\]
	We first prove that $H^b(U, \sC^q(\frU,\sF))=0$ if $b\neq0$.
	If we fix some total order on the set $A$, then we have 
	\[
		\sC^q(\frU,\sF)\cong\prod_{\alpha_0<\cdots<\alpha_q} j_{\bsalpha*}j^*_\bsalpha\sF,
	\]	
	and each component $j_{\bsalpha*}j^*_\bsalpha\sF$ is acyclic for the functor $\Gamma(U,-)$
	since $U_\bsalpha$ is affine.
	Therefore $\sC^q(\frU,\sF)$ is acyclic by Lemma \ref{lem: 3.5} below.
	It is now sufficient to prove that $H^a\bigl(\Delta, H^0(U, \sC^q(\frU,\sF))\bigr)=0$ for any integer $a\neq0$, where
	\[
		 H^0(U, \sC^q(\frU,\sF)) =  \prod_{\bsalpha\in A^{q+1}}^\alt \Gamma(U,j_{\bsalpha*}j_\bsalpha^*\sF)
		  \cong\prod_{\alpha_0<\cdots<\alpha_q} \Gamma(U_\bsalpha,\sF).
	\]
	Assume that the total order on $A$ is preserved by the action of $\Delta$
	(for example, we may take the order on $\bbR$ through an embedding $\tau\colon A\hookrightarrow\bbR$
	for some $\tau\in I$). 
	Let $B$ be the subset of $A^{q+1}$ consisting of elements $\bsalpha = (\alpha_0,\ldots,\alpha_q)$
	such that $\alpha_0<\cdots<\alpha_q$.
	Then action of $\Delta$ on $B$ is free.
	We denote by $B_0$ a subset of $B$ representing the set 
	$\Delta\backslash B$,
	so that any $\bsalpha\in B$ may be written uniquely as $\bsalpha=\varepsilon\bsalpha_0$ for some 
	$\varepsilon\in\Delta$ and $\bsalpha_0\in B_0$.  We let
	\[
		M\coloneqq\prod_{\bsalpha\in B_0} \Gamma(U_{\bsalpha}, \sF),
	\]
	and we let
	$
		\Hom_\bbZ(\bbZ[\Delta], M)
	$
	be the coinduced module of $M$,
	with the action of $\Delta$ given for any $\varphi\in\Hom_\bbZ(\bbZ[\Delta], M)$
	by $\varepsilon\varphi(u)=\varphi(u\varepsilon)$ for any $u\in\bbZ[\Delta]$ and $\varepsilon\in\Delta$.
	Then we have a $\bbZ[\Delta]$-linear isomorphism
	\begin{equation}\label{eq: isomorphism vulcan}
		H^0(U, \sC^q(\frU,\sF)) \xrightarrow\cong \Hom_\bbZ(\bbZ[\Delta], M)
	\end{equation}
	given by mapping any $(s_\bsalpha)\in H^0(U, \sC^q(A,\sF))$ to the $\bbZ$-linear homomorphism
	\[
		\varphi_{(s_\bsalpha)}(\delta)\coloneqq\Bigl(\iota_\delta\bigl([\delta]^*s_{\delta^{-1}\bsalpha_0}\bigr)\Bigr)\in M
	\]
	for any $\delta\in\Delta$.
	The compatibility of \eqref{eq: isomorphism vulcan} with the action of $\Delta$ is seen as follows.
	By definition, the action of $\varepsilon\in\Delta$ on $(s_\bsalpha)\in H^0(U,\sC^q(A,\sF))$ is given by
	$\varepsilon\bigl((s_\bsalpha)\bigr) 
	=\bigl(\iota_\varepsilon([\varepsilon]^* s_{\varepsilon^{-1}\bsalpha})\bigr)$.
	Hence noting that 
	\[
		\iota_\delta\circ[\delta]^*\iota_\varepsilon= \iota_{\delta\varepsilon}\colon\Gamma(U_\bsalpha,[\delta\varepsilon]^*\sF)
		\rightarrow\Gamma(U_\bsalpha,\sF)
	\]
	and
	$
		[\delta]^* \circ \iota_\varepsilon= [\delta]^*\iota_\varepsilon \circ [\delta]^*\colon\Gamma(U_\bsalpha,[\varepsilon]^*\sF)
		\rightarrow\Gamma(U_\bsalpha,[\delta]^*\sF)
	$
	for any $\delta\in\Delta$, we have
	\begin{align*}
			\varphi_{\varepsilon(s_\bsalpha)}(\delta)&=
			\Bigl(\iota_\delta\bigl([\delta]^*(\iota_\varepsilon([\varepsilon]^*s_{\varepsilon^{-1}\delta^{-1}\bsalpha_0}))\bigr)\Bigr)
			=\Bigl((\iota_\delta\circ[\delta]^*\iota_\varepsilon)\bigl([\delta\varepsilon]^*s_{\varepsilon^{-1}\delta^{-1}\bsalpha_0}\bigr)\Bigr)\\
			&=\bigl(\iota_{\delta\varepsilon}([\delta\varepsilon]^*s_{\varepsilon^{-1}\delta^{-1}\bsalpha_0})\bigr) 
			=\varphi_{(s_\bsalpha)}(\delta\varepsilon)
	\end{align*}
	as desired.  The fact that \eqref{eq: isomorphism vulcan} is an isomorphism
	follows from the fact that $B_0$ is a representative of $\Delta\backslash B$. 
	By \eqref{eq: isomorphism vulcan} and Shapiro's lemma, we have $H^a(\Delta, H^0(U,\sC^q(A,\sF)))
	\cong H^a(\{1\},M)=0$ for $a\neq0$ as desired.
\end{proof}

The following Lemma \ref{lem: 3.5} was used in the proof of Proposition \ref{proposition: acyclic}.

\begin{lemma}\label{lem: 3.5}
	Let $I$ be a scheme and let $(\sF_\lambda)_{\lambda\in\Lambda}$ be a family of quasi-coherent sheaves on $I$.
	Then for any integer $m\geq0$, we have
	\[
		H^m\biggl(I, \prod_{\lambda\in\Lambda}\sF_\lambda\biggr) \cong\prod_{\lambda\in\Lambda}H^m(I, \sF_\lambda).
	\]
\end{lemma}

\begin{proof}
	Take an injective resolution $0\rightarrow\sF_\lambda\rightarrow I^\bullet_\lambda$ for each $\lambda\in\Lambda$.
	We will prove that $0\rightarrow\prod_{\lambda\in\Lambda}\sF_\lambda\rightarrow \prod_{\lambda\in\Lambda}I^\bullet_\lambda$
	is an injective resolution.
	Since the product of injective objects is injective, it is sufficient to prove that $0 \rightarrow 
	\prod_{\lambda\in\Lambda}\sF_\lambda\rightarrow \prod_{\lambda\in\Lambda}I^\bullet_\lambda$ is exact.

	For any affine open set $V$ of $I$, by affine vanishing, the global section
	$0 \rightarrow \sF_\lambda(V) \rightarrow I^\bullet_\lambda(V)$ is exact, hence the product 
	\begin{equation}\label{eq: orange}
		0 \rightarrow\prod_{\lambda\in\Lambda}\sF_\lambda(V) \rightarrow
		\prod_{\lambda\in\Lambda} I^\bullet_\lambda(V)
	\end{equation}
	is also exact.
	For any $x\in I$, if we take the direct limit of \eqref{eq: orange} with respect to open affine neighborhoods of $x$,
	then we obtain the exact sequence
	\[
	0 \rightarrow 
		\Biggl(\prod_{\lambda\in\Lambda}\sF_\lambda\Biggr)_x\rightarrow
		\Biggl( \prod_{\lambda\in\Lambda}I^\bullet_\lambda\Biggr)_x.
	\]
	This shows that 
	 $0 \rightarrow 
	\prod_{\lambda\in\Lambda}\sF_\lambda\rightarrow \prod_{\lambda\in\Lambda}I^\bullet_\lambda$ is exact
	as desired.
\end{proof}

Proposition \ref{proposition: acyclic} gives the following Corollary.

\begin{corollary}\label{cor: description}
	We let $C^\bullet(\frU/\Delta,\sF)\coloneqq\Gamma(U/\Delta,\sC^\bullet(\frU,\sF))$.
	Then for any integer $m\geq0$, the equivariant cohomology $H^m(U/\Delta,\sF)$
	is given as 
	\[
		H^m(U/\Delta,\sF) = H^m(C^\bullet(\frU/\Delta,\sF)).
	\]	
\end{corollary}

By definition, for any integer $q\in\bbZ$, we have
\[
	C^q(\frU/\Delta,\sF)= \biggl(\prod^\alt_{\bsalpha\in A^{q+1}}\Gamma(U_\bsalpha,\sF)\biggr)^\Delta.
\]

%
%
%
\section{Shintani Generating Class}\label{section: Shintani Class}
%
%
%

We let $\bbT$ be the algebraic torus of \eqref{eq: algebraic torus}, and let $U=\bbT\setminus\{1\}$.
In this section, we will use the description of equivariant cohomology of Corollary \ref{cor: description}
to define the Shintani generating class as a class in $H^{g-1}(U/\Delta,\sO_\bbT)$.  
We will then consider the action of the differential operators $\partial_\tau$ on this class.

We first interpret the generating functions $\sG_\sigma(z)$ of Definition \ref{def: generating}
as rational functions on $\bbT$.
Let $\frD^{-1}\coloneqq\{u\in F\mid\Tr_{F/\bbQ}(u\cO_F)\subset\bbZ\}$ be the inverse different of $F$.
Then there exists an isomorphism
\begin{equation}\label{eq: uniformization}
	(F\otimes\bbC)/\frD^{-1} \xrightarrow\cong\bbT(\bbC) = \Hom_\bbZ(\cO_F,\bbC^\times),  \qquad z  \mapsto \xi_z
\end{equation}
given by mapping any $z\in F\otimes\bbC$ to the character $\xi_z(\alpha)\coloneqq e^{2\pi i \Tr(\alpha z)}$ in 
$\Hom_\bbZ(\cO_F,\bbC^\times)$.
The function $t^\alpha$ on $\bbT(\bbC)$ corresponds through  the isomorphism \eqref{eq: uniformization} to the function
$
	e^{2\pi i \Tr(\alpha z)}
$
for any $\alpha\in\cO_F$.  
Thus the following holds.

\begin{lemma}\label{lem: correspondence}
	For $\bsalpha=(\alpha_1,\ldots,\alpha_g)\in A^g$ and $\sigma\coloneqq\sigma_\bsalpha$,
	consider the rational function
	\[
		\cG_\sigma(t)\coloneqq\frac{\sum_{\alpha\in \wh P_\bsalpha\cap\cO_F}t^\alpha}{{\bigl(1-t^{\alpha_1}\bigr)\cdots
		\bigl(1-t^{\alpha_g}\bigr)}}
	\]
	on $\bbT$, where $P_\bsalpha$ is again the parallelepiped spanned by $\alpha_1,\ldots,\alpha_g$.
	Then $\cG_\sigma(t)$ corresponds to the function $\sG_{\sigma}(z)$ of 
	Definition \ref{def: generating} through the uniformization \eqref{eq: uniformization}.
	Note that by definition, if we let $B\coloneqq\bbZ[t^{\alpha}\mid\alpha\in\cO_{F+}]$,  then we have
	\begin{equation}\label{eq: B}
		\cG_\sigma(t)\in B_\bsalpha\coloneqq 
		B\Bigl[\frac{1}{1-t^{\alpha_1}},\ldots,\frac{1}{1-t^{\alpha_g}}\Bigr].
	\end{equation}
\end{lemma}

Again, we fix an ordering $I=\{\tau_1,\ldots,\tau_g\}$.  For any $\bsalpha=(\alpha_1,\ldots,\alpha_g)\in\cO_{F+}^g$, 
let $\bigl(\alpha_j^{\tau_i}\bigr)$ be the matrix in $M_g(\bbR)$ whose $(i,j)$-component is $\alpha_j^{\tau_i}$.
We let $\sgn(\bsalpha)\in\{0,\pm1\}$ be the signature of  $\det\bigl(\alpha_j^{\tau_i}\bigr)$.
We define the \textit{Shintani generating class} $\cG$ as follows.

\begin{proposition}\label{prop: Shintani generating class}
	For any $\bsalpha=(\alpha_1,\ldots,\alpha_g)\in A^{g}$, we let
	\[
		\cG_\bsalpha\coloneqq\sgn(\bsalpha)\cG_{\sigma_\bsalpha}(t)
		\in\Gamma(U_\bsalpha,\sO_\bbT).
	\]
	Then we have
	$
		(\cG_\bsalpha) \in C^{g-1}(A,\sO_\bbT).
	$
	Moreover, $(\cG_\bsalpha)$ satisfies the cocycle condition $d^{g-1}(\cG_\bsalpha)=0$, hence defines a class
	\[
		\cG\coloneqq [\cG_\bsalpha]\in H^{g-1}(U/\Delta,\sO_\bbT).
	\]	
	We call this class the Shintani generating class.
\end{proposition}

\begin{proof}
	By construction, $(\cG_\bsalpha)$ defines 
	an element in $\Gamma(U, \sC^{g-1}(\frU,\sO_\bbT))=\prod^\alt_{\bsalpha\in A^g} \Gamma(U_\bsalpha,\sO_\bbT)$.
	Since $[\varepsilon]^*\cG_\bsalpha=\cG_{\varepsilon\bsalpha}$ for any $\varepsilon\in\Delta$,
	we have
	\[
		(\cG_\bsalpha)\in \Gamma\bigl(U, \sC^{g-1}(\frU,\sO_\bbT)\bigr)^{\Delta}= C^{g-1}(\frU/\Delta,\sO_\bbT).
	\]
	To prove the cocycle condition $d^{g-1}(\cG_\bsalpha)=0$, it is sufficient to check that
	\begin{equation}\label{eq: cocycle condition}
		\sum_{j=0}^g(-1)^j\cG_{(\alpha_0,\ldots,\check\alpha_j,\ldots\alpha_g)}=0
	\end{equation}
	for any $\alpha_0,\ldots,\alpha_g\in A$.
	By definition, the rational function $\cG_{\sigma_\bsalpha}(t)$ maps to the formal power series
	\[
		\cG_{\sigma_\bsalpha}(t) =\sum_{\alpha\in\wh\sigma_\bsalpha\cap\cO_F}t^\alpha
	\]
	by taking the formal completion 
	$B_\bsalpha\hookrightarrow\bbZ\llbracket t^{\alpha_1},\ldots,t^{\alpha_g}\rrbracket$, where $B_\bsalpha$ is the ring defined in \eqref{eq: B}.
	Since the map taking the formal completion is injective,
	it is sufficient to check \eqref{eq: cocycle condition} for the associated formal power series.
	By \cite{Yam10}*{Proposition 6.2}, we have
	\[
		\sum_{j=0}^g(-1)^j\sgn(\alpha_0,\ldots,\check\alpha_j,\ldots,\alpha_g)
		\boldsymbol{1}_{\wh\sigma_{(\alpha_0,\ldots,\check\alpha_j,\ldots\alpha_g)}}\equiv0
	\]
	as a function on $\bbR_+^I$, where $\boldsymbol{1}_R$ is the characteristic function of $R\subset\bbR^I_+$ satisfying
	$\boldsymbol{1}_R(x)=1$ if $x\in R$ and $\boldsymbol{1}_R(x)=0$ if $x\not\in R$.
	Our assertion now follows by examining the formal power series expansion of 
	$\cG_{(\alpha_0,\ldots,\check\alpha_j,\ldots,\alpha_g)}$.
\end{proof}

We will next define differential operators $\partial_\tau$ for $\tau\in I$ on equivariant cohomology.
Since $t^\alpha = e^{2\pi i \Tr(\alpha z)}$ through \eqref{eq: uniformization} for any $\alpha\in\cO_F$, we have
\begin{equation}\label{eq: relation}
	\frac{dt^\alpha}{t^\alpha} = \sum_{\tau\in I} 2\pi i \alpha^\tau dz_\tau.
\end{equation}
Let $\alpha_1,\ldots,\alpha_g$ be a basis of $\cO_F$.
For any $\tau\in I$, we let $\partial_\tau $ be the differential operator
\[
	\partial_\tau \coloneqq\sum_{j=1}^g  \alpha_j^\tau  t^{\alpha_j}\frac{\partial}{\partial t^{\alpha_j}}.
\]
By \eqref{eq: relation}, we see that $\partial_\tau $ corresponds to the differential operator
$\frac{1}{2\pi i}\frac{\partial}{\partial z_\tau}$
through the uniformization \eqref{eq: uniformization}, and hence is independent of the choice of the basis $\alpha_1,\ldots,\alpha_g$.
By Theorem \ref{theorem: Shintani} and Lemma \ref{lem: correspondence}, we have the following result.

\begin{proposition}\label{prop: Shintani}
	Let $\bsalpha=(\alpha_1,\ldots,\alpha_g)\in A^g$ and $\sigma=\sigma_\bsalpha$.
	For any $\bsk=(k_\tau)\in\bbN^I$ and $\partial^\bsk\coloneqq\prod_{\tau\in I}\partial_\tau^{k_\tau}$, 
	we have
	\[
		\partial^\bsk\cG_\sigma(\xi) = \zeta_\sigma(\xi,-\bsk)
	\]
	for any torsion point $\xi\in U_{\bsalpha}$.
\end{proposition}

The differential operator $\partial_\tau$ gives a morphism of abelian sheaves
\[
	\partial_\tau\colon\sO_{\bbT_{F^\tau}}(\bsk)\rightarrow\sO_{\bbT_{F^\tau}}(\bsk-1_\tau)
\]
compatible with the action of $\Delta$ for any $\bsk\in\bbZ^I$,
where $\bbT_{F^\tau}\coloneqq\bbT\otimes F^\tau$
is the base change of $\bbT$ to $F^\tau$.
This induces a homomorphism
\[
	\partial_\tau\colon H^m(U_{F^\tau}/\Delta,\sO_{\bbT_{F^\tau}}(\bsk))\rightarrow 
	H^m(U_{F^\tau}/\Delta,\sO_{\bbT_{F^\tau}}(\bsk-1_\tau))
\]
on equivariant cohomology.

\begin{lemma}\label{lem: differential}
	Let $\wt F$ be the composite of $F^\tau$ for all $\tau\in I$.
	The operators $\partial_\tau$ for $\tau\in I$, considered over $\wt F$, are commutative with each other.
	Moreover, the composite 
	\[
		\partial\coloneqq\prod_{\tau\in I}\partial_\tau\colon 
		\sO_{\bbT_{\wt F}} \rightarrow \sO_{\bbT_{\wt F}}(1,\ldots,1)=\sO_{\bbT_{\wt F}}
	\]
	is defined over $\bbQ$,
	that is, it is a base change to $\wt F$ of a morphism of abelian sheaves
	$
		\partial\colon \sO_\bbT\rightarrow\sO_\bbT.
	$
	In particular, $\partial$ induces a homomorphism 
	\begin{equation}\label{eq: differential}
		\partial\colon H^m(U/\Delta,\sO_\bbT) \rightarrow H^m(U/\Delta,\sO_\bbT).
	\end{equation}
\end{lemma}

\begin{proof}
	The commutativity is clear from the definition.
	Since the group $\Gal(\wt F/\bbQ)$ permutes the operators $\partial_\tau$,
	the operator $\partial$ is invariant under this action.  This gives our assertion.
\end{proof}

Our main result, which we prove in \S \ref{section: specialization}, concerns the specialization of the classes
\begin{equation}\label{eq: main}
	\partial^k\cG  \in H^{g-1}\bigl(U/\Delta,\sO_\bbT\bigr)
\end{equation}
for $k\in\bbN$ at nontrivial torsion points of $\bbT$.

%
%
%
\section{Specialization to Torsion Points}\label{section: specialization}
%
%
%

For any nontrivial torsion point $\xi$ of $\bbT$, let $\Delta_\xi\subset\Delta$ be the isotropic subgroup
of $\xi$.  Then we may view $\xi\coloneqq\Spec\bbQ(\xi)$ as a $\Delta_\xi$-scheme with a trivial action of $\Delta_\xi$.
Then the natural inclusion $\xi\rightarrow U$ for $U\coloneqq\bbT\setminus\{1\}$
is compatible with the inclusion $\Delta_\xi\subset\Delta$,
hence the pullback \eqref{eq: pullback} induces the specialization map
\[
	\xi^*\colon H^m(U/\Delta,\sO_\bbT)\rightarrow H^m(\xi/\Delta_\xi,\sO_\xi).
\]
The purpose of this section is to prove our main theorem, given as follows.

\begin{theorem}\label{theorem: main}
	Let $\xi$ be a nontrivial torsion point of $\bbT$, and let $k$ be an integer $\geq0$.
	If we let $\cG$ be the Shintani generating class defined in Proposition \ref{prop: Shintani generating class},
	and if we let
	$\partial^k\cG(\xi)\in H^{g-1}(\xi/\Delta_\xi,\sO_\xi)$ be image 
	by the specialization map $\xi^*$ of the class $\partial^k\cG$
	defined in \eqref{eq: main}, then we have
	\[
		\partial^k\cG(\xi) = \cL(\xi\Delta,-k)
	\]
	through the isomorphism $H^{g-1}(\xi/\Delta_\xi,\sO_\xi)\cong\bbQ(\xi)$ given in Proposition \ref{prop: ecc} below.
\end{theorem}

We will prove Theorem \ref{theorem: main} at the end of this section.
The specialization map can be expressed explicitly in terms of 
cocycles as follows.  
We let $V_\alpha\coloneqq U_\alpha\cap \xi$ for any $\alpha\in A$.
Then $\frV\coloneqq\{V_\alpha\}_{\alpha\in A}$ is an affine open covering of $\xi$.
For any integer $q\geq 0$ and $\bsalpha=(\alpha_0,\ldots,\alpha_{q})\in A^{q+1}$,
we let  $V_\bsalpha\coloneqq V_{\alpha_0}\cap\cdots\cap V_{\alpha_{q}}$ and
\begin{equation}\label{eq: alternating}
	C^q\bigl(\frV/\Delta_\xi,\sO_{\xi}\bigr)\coloneqq
	\biggl(\prod_{\bsalpha\in A^{q+1}}^\alt\Gamma(V_\bsalpha,\sO_{\xi})\biggr)^{\Delta_\xi}.
\end{equation}
Here, note that $\Gamma(V_\bsalpha,\sO_\xi)=\bbQ(\xi)$ if $V_\bsalpha\neq\emptyset$ and 
$\Gamma(V_\bsalpha,\sO_\xi)=\{0\}$ otherwise.
The same argument as that of Corollary \ref{cor: description} shows that
we have
\begin{equation}\label{eq: similarly}
	H^m(\xi/\Delta_\xi,\sO_\xi)\cong H^m\bigl(C^\bullet\bigl(\frV/\Delta_\xi,\sO_{\xi}\bigr)\bigr).
\end{equation}

We let $A_\xi$ be the subset of elements $\alpha\in A$ satisfying $\xi\in U_\alpha$.
This is equivalent to the condition that $\xi(\alpha)\neq 1$.
We will next prove in Lemma \ref{lemma: 5.2} 
that the cochain complex $C^\bullet\bigl(\frV/\Delta_\xi,\sO_{\xi}\bigr)$ of \eqref{eq: alternating}
is isomorphic to the dual of the chain complex 
$C_\bullet(A_\xi)$ defined as follows.
For any integer $q\geq0$, we let 
\[
	C_q(A_\xi)\coloneqq\bigoplus^\alt_{\bsalpha\in A_\xi^{q+1}}\bbZ\bsalpha
\]
be the quotient of $\bigoplus_{\bsalpha\in A_\xi^{q+1}}\bbZ\bsalpha$ by
the submodule generated by 
\[
	\{\rho(\bsalpha)-\sgn(\rho)\bsalpha\mid \bsalpha\in A_\xi^{q+1}, \rho\in\frS_{q+1}\}
	\cup \{ \bsalpha=(\alpha_0,\ldots,\alpha_q) \mid  \text{$\alpha_i=\alpha_j$ for some $i\neq j$}\}.
\]
We denote by $\langle\bsalpha\rangle$ the class represented by $\bsalpha$ in $C_q(A_\xi)$.
We see that $C_q(A_\xi)$ has a natural action of $\Delta_\xi$ and is a free $\bbZ[\Delta_\xi]$-module.
In fact, a basis of $C_q(A_\xi)$ may be constructed in a similar way to the construction of $B_0$
in the proof of Proposition \ref{proposition: acyclic}.
Then $C_\bullet(A_\xi)$ is a complex of $\bbZ[\Delta_\xi]$-modules with respect to the standard differential operator
$d_q\colon C_q(A_\xi)\rightarrow C_{q-1}(A_\xi)$  given by
\[
	d_q(\langle\alpha_0,\ldots,\alpha_{q}\rangle)\coloneqq\sum_{j=0}^{q}
	(-1)^{j}\langle\alpha_0,\ldots,\check\alpha_{j},\ldots,\alpha_{q}\rangle
\]
for any $\bsalpha=(\alpha_0,\ldots,\alpha_{q})\in A_\xi^{q+1}$.
If we let $d_0\colon C_0(A_\xi) \rightarrow \bbZ$ be the homomorphism defined
by $d_0(\langle \alpha\rangle)\coloneqq 1$ for any $\alpha\in A_\xi$, then 
$C_\bullet(A_\xi)$ is a free resolution  of $\bbZ$ with trivial $\Delta_\xi$-action.
We have the following.

\begin{lemma}\label{lemma: 5.2}
	There exists a natural isomorphism of complexes
	\[
		C^\bullet(\frV/\Delta_\xi,\sO_{\xi})\xrightarrow\cong \Hom_{\Delta_\xi}(C_\bullet(A_\xi),\bbQ(\xi)).
	\]
\end{lemma}

\begin{proof}
	The natural isomorphism
	\[
		\prod_{\bsalpha\in A^{q+1}}\Gamma(V_\bsalpha,\sO_\xi)=\prod_{\bsalpha\in A^{q+1}_\xi}\bbQ(\xi)
		\cong\Hom_\bbZ\left(\bigoplus_{\bsalpha\in A_\xi^{q+1}}\bbZ\bsalpha,\bbQ(\xi)\right)
	\]
	induces an isomorphism between the submodules
	\[
		C^q(\frV/\Delta_\xi,\sO_\xi)=\left(\prod^\alt_{\bsalpha\in A^{q+1}}\Gamma(V_\bsalpha,\sO_\xi)  \right)^{\Delta_\xi}
		\subset 	\prod_{\bsalpha\in A^{q+1}}\Gamma(V_\bsalpha,\sO_\xi)
	\]
	and 
	\[
		\Hom_{\Delta_\xi}(C_q(A_\xi),\bbQ(\xi))\subset\Hom_\bbZ
		\left(\bigoplus_{\bsalpha\in A_\xi^{q+1}}\bbZ\bsalpha,\bbQ(\xi)\right).
	\]
	Moreover, this isomorphism is compatible with the differential.  
\end{proof}
We will next use a Shintani decomposition (see Definition \ref{def: Shintani}) to
construct a complex which is quasi-isomorphic to the complex $C_\bullet\bigl(A_\xi)$.

\begin{lemma}\label{lem: inclusion}
	Let $\xi$ be as above.
	There exists a Shintani decomposition $\Phi$ such that any $\sigma\in\Phi$
	is of the form
	$\sigma_\bsalpha=\sigma$
	for some $\bsalpha\in A_\xi^{q+1}$.
\end{lemma}

\begin{proof}
	Let $\Phi$ be a Shintani decomposition.  We will deform $\Phi$ to construct a Shintani decomposition
	satisfying our assertion.  
	Let $\Lambda$ be a finite subset of $A$ such that $\{\sigma_\alpha \mid \alpha\in \Lambda\}$
	 represents the quotient set $\Delta_\xi\backslash\Phi_1$.	
	 If $\xi(\alpha)\neq 1$ for any $\alpha\in\Lambda$, then $\Phi$ satisfies our assertion since 
	 $\alpha\in A_\xi$ if and only if $\xi(\alpha)\neq 1$.
	 
	 Suppose that there exists $\alpha\in \Lambda$ such that $\xi(\alpha)=1$.
	 Since $\xi$ is a nontrivial character on $\cO_F$, there exists $\beta\in\cO_{F+}$ such that $\xi(\beta)\neq 1$.
	 Then for any integer $N$, we have $\xi(N\alpha+\beta)\neq 1$. 
	 Let $\Phi'$ be the set of cones obtained by deforming $\sigma=\sigma_\alpha$ to $\sigma'\coloneqq\sigma_{N\alpha+\beta}$
	 and $\varepsilon\sigma$ to $\varepsilon\sigma'$ for any $\varepsilon\in\Delta_\xi$.
	By taking $N$ sufficiently large, the amount of deformation can be made arbitrarily small so that
	$\Phi'$ remains a fan.
	By repeating this process, we obtain a Shintani decomposition satisfying the desired condition.
\end{proof}

In what follows, we fix a Shintani decomposition $\Phi$ satisfying the condition of Lemma \ref{lem: inclusion}.
Let $N\colon\bbR_+^I\rightarrow\bbR_+$ be the norm map defined by 
$N((a_\tau))\coloneqq\prod_{\tau\in I}a_\tau$,
and we let 
\[
	\bbR_1^I\coloneqq\{(a_\tau)\in\bbR_+^I\mid N((a_\tau))=1\}
\]
be the subset of $\bbR^I_+$ of norm one.
For any $\sigma\in\Phi_{q+1}$, the intersection $\sigma\cap\bbR_1^I$
is a subset of $\bbR_1^I$ which is homeomorphic to a simplex of dimension $q$, 
and the set $\{  \sigma\cap \bbR_1^I  \mid \sigma\in\Phi_+\}$ 
for $\Phi_+\coloneqq\bigcup_{q\geq0}\Phi_{q+1}$
gives a simplicial decomposition of the topological 
space $\bbR_1^I$.

In what follows,  for any $\sigma\in\Phi_{q+1}$, we denote by $\langle\sigma\rangle$ the class $\sgn(\bsalpha)\langle\bsalpha\rangle$ in 
$C_q(A_\xi)$, where $\bsalpha\in A_\xi^{q+1}$ is a generator of $\sigma$.  Recall that such a 
generator $\bsalpha$ is uniquely determined up to permutation from $\sigma$.
We then have the following.

\begin{lemma}\label{lemma: 5.4}
	For any integer $q\geq0$, we let $C_q(\Phi)$ be the $\bbZ[\Delta_\xi]$-submodule of $C_q(A_\xi)$
	generated by $\langle\sigma\rangle$ for all $\sigma\in\Phi_{q+1}$.
	Then $C_\bullet(\Phi)$ is a subcomplex of $C_\bullet(A_\xi)$
	which also gives a free resolution of $\bbZ$ as a $\bbZ[\Delta_\xi]$-module.
	In particular, the natural inclusion induces a quasi-isomorphism of complexes
	\[
		C_\bullet(\Phi)\xrightarrow{\qis} C_\bullet(A_\xi)
	\]
	compatible with the action of $\Delta_\xi$.
\end{lemma}

\begin{proof}
	Note that $C_q(\Phi)$ for any integer $q\geq0$
	is a free $\bbZ[\Delta_\xi]$-module
	generated by representatives of the quotient $\Delta_\xi\backslash \Phi_{q+1}$.
	By construction, $C_\bullet(\Phi)$ can be identified with the chain complex associated to 
	the simplicial decomposition $\{\sigma\cap\bbR^I_1\mid \sigma\in\Phi_+\}$
	of the topological space $\bbR^I_1$, hence we see that the complex $C_\bullet(\Phi)$ is exact and gives a free resolution
	of $\bbZ$ as a $\bbZ[\Delta_\xi]$-module.  Our assertion follows from the fact that $C_\bullet(A_\xi)$ also gives a free
	resolution of $\bbZ$ as a $\bbZ[\Delta_\xi]$-module.
\end{proof}

We again fix a numbering of elements in $I$ so that $I=\{\tau_1,\ldots,\tau_g\}$.
We let
\[
	L\colon\bbR_+^I\rightarrow\bbR^g
\]
be the homeomorphism defined by $(x_\tau)\mapsto(\log x_{\tau_i})$.   If we let $W\coloneqq\{ (y_{\tau_i})\in\bbR^g
\mid\sum_{i=1}^gy_{\tau_i}=0\}$, then $W$ is an $\bbR$-linear subspace of $\bbR^g$ of dimension $g-1$, and $L$ gives a 
homeomorphism $\bbR_1^I\cong W\cong \bbR^{g-1}$.
For $\Delta_\xi\subset F$,  the Dirichlet unit theorem (see for example \cite{Sam70}*{Theorem 1 p.61})
shows that the discrete subset $L(\Delta_\xi)\subset W$ is a free $\bbZ$-module of rank $g-1$,
hence we have
\[
	\cT_\xi\coloneqq\Delta_\xi\backslash\bbR_1^I\cong \bbR^{g-1}/\bbZ^{g-1}.
\]

We consider the coinvariant
\[
	C_q(\Delta_\xi\backslash\Phi)\coloneqq  C_q(\Phi)_{\Delta_\xi}
\]
of $C_q(\Phi)$ with respect to the action of $\Delta_\xi$,
that is, the quotient of  $C_q(\Phi)$ by the subgroup generated by $\langle\sigma\rangle-\langle\varepsilon\sigma\rangle$
for $\sigma\in\Phi_{q+1}$ and $\varepsilon\in\Delta_\xi$.
For any $\sigma\in\Phi_{q+1}$, we denote by $\ol\sigma$ the image of $\sigma$ in the quotient 
$\Delta_\xi\backslash\Phi_{q+1}$, and we denote by $\langle\ol\sigma\rangle$
the image of $\langle\sigma\rangle$
in $C_q(\Delta_\xi\backslash\Phi)$,
which depends only on the class $\ol\sigma\in
\Delta_\xi\backslash\Phi_{q+1}$.
Then the set $\{ \Delta_\xi  \backslash(\Delta_\xi\sigma\cap\bbR^I_1)
\mid \ol\sigma\in \Delta_\xi\backslash\Phi_+\}$ of subsets of $\cT_\xi$
gives a simplicial decomposition of $\cT_\xi$ and
$C_\bullet(\Delta_\xi\backslash \Phi)$
may be identified with the associated chain complex.
Hence we have
\begin{align*}
	H_m(C_\bullet(\Delta_\xi\backslash\Phi))&= H_m(\cT_\xi,\bbZ),&
	H^m\Bigl(\Hom_\bbZ\bigl(C_\bullet(\Delta_{\xi}\backslash\Phi), \bbZ\bigr)\Bigr)&= H^m(\cT_\xi,\bbZ).
\end{align*}
Since $\cT_\xi\cong\bbR^{g-1}/\bbZ^{g-1}$,
the homology groups $H_m(\cT_\xi,\bbZ)$ for integers $m$ are free abelian groups,
and the pairing
 \begin{equation}\label{eq: pairing}
 	H_m(\cT_\xi,\bbZ) \times H^m(\cT_\xi,\bbZ) \rightarrow \bbZ,
 \end{equation}
obtained by associating to a cycle
 $u\in C_m(\Delta_\xi\backslash\Phi)$ and a cocycle $\varphi \in \Hom_\bbZ\bigl(C_m(\Delta_\xi\backslash\Phi), \bbZ\bigr)$
 the element $\varphi(u)\in\bbZ$, is perfect 
  (see for example \cite{Mun84}*{Theorem 45.8}).

The generator of the cohomology group
\[
	 H_{g-1}(\cT_\xi,\bbZ) = H_{g-1}\bigl(C_\bullet(\Delta_\xi\backslash\Phi)\bigr) \cong \bbZ
\]
is given by the fundamental class
\begin{equation}\label{eq: fundamental class}
	\sum_{\ol\sigma\in\Delta_\xi\backslash\Phi_{g}} \langle\ol\sigma\rangle
	\in C_{g-1}(\Delta_\xi\backslash\Phi),
\end{equation}
and the canonical isomorphism
\begin{equation}\label{eq: isom}
	H^{g-1}(\cT_\xi,\bbQ(\xi)) = H^{g-1}\Bigl(\Hom_\bbZ\bigl(C_\bullet(\Delta_\xi\backslash\Phi), \bbQ(\xi)\bigr)\Bigr) \cong \bbQ(\xi)
\end{equation}
induced by the fundamental class \eqref{eq: fundamental class} via the pairing \eqref{eq: pairing} 
is given explicitly in terms of cocycles
by mapping any $\varphi\in \Hom_\bbZ(C_{g-1}(\Delta_\xi\backslash\Phi), \bbQ(\xi))$ to the element 
$\sum_{\ol\sigma\in\Delta_\xi\backslash\Phi_g}\varphi(\langle\ol\sigma\rangle)\in \bbQ(\xi)$.

\begin{proposition}\label{prop: ecc}
	Let $\eta\in H^{g-1}(\xi/\Delta_\xi,\sO_\xi)$ be represented by a cocycle 
	\[
		 (\eta_\bsalpha)\in C^{g-1}(\frV/\Delta_\xi,\sO_\xi)
		=\biggl(\prod_{\bsalpha\in A_\xi^{g}}^\alt
	\bbQ(\xi)\biggr)^{\Delta_\xi}.
	\]	
	For any cone $\sigma\in\Phi_g$, let 
	$\eta_\sigma\coloneqq \sgn(\bsalpha)\eta_{\bsalpha}$ for any $\bsalpha\in A^g_\xi$ such that $\sigma_\bsalpha=\sigma$.
	Then the homomorphism mapping the cocycle $(\eta_\bsalpha)$ to
	$
		\sum_{\ol\sigma\in\Delta_\xi\backslash\Phi_{g}} \eta_{\sigma}
	$
	induces a canonical isomorphism
	\begin{equation}\label{eq: isom main}
		H^{g-1}(\xi/\Delta_\xi,\sO_\xi)\cong\bbQ(\xi).
	\end{equation}
\end{proposition}

\begin{proof}
	Since $C_q(\Phi)$ and $C_q(A_\xi)$ are free $\bbZ[\Delta_\xi]$-modules, 
	the quasi-isomorphism $C_\bullet(\Phi)\xrightarrow\qis C_\bullet(A_\xi)$ of Lemma \ref{lemma: 5.4} induces the 
	quasi-isomorphism
	\[
	\Hom_{\Delta_\xi}\bigl(C_\bullet(A_\xi),\bbQ(\xi)\bigr)\xrightarrow\qis \Hom_{\Delta_\xi}\bigl(C_\bullet(\Phi),\bbQ(\xi)\bigr).
	\]	
	Combining this fact with Lemma \ref{lemma: 5.2} and \eqref{eq: similarly}, we see that
	\[	
		H^{g-1}(\xi/\Delta_\xi,\sO_\xi)\cong H^{g-1}\bigl(\Hom_{\Delta_\xi}\bigl(C_\bullet(\Phi),\bbQ(\xi)\bigr)\bigr).
	\]
	 Since we have
	$
		\Hom_{\Delta_\xi}\bigl(C_\bullet(\Phi),\bbQ(\xi)\bigr)
		= \Hom_{\bbZ}\bigl(C_\bullet(\Delta_\xi\backslash\Phi),\bbQ(\xi)\bigr),
	$
	our assertion follows from \eqref{eq: isom}.
\end{proof}

We will now prove Theorem \ref{theorem: main}.

\begin{proof}[Proof of Theorem \ref{theorem: main}]
	By construction 
	and Lemma \ref{lem: differential}, the class $\partial^k\cG(\xi)$ is a class defined over $\bbQ(\xi)$ represented by the cocycle
	$
		(\partial^k\cG_\bsalpha(\xi)) \in C^{g-1}(\frV/\Delta_\xi,\sO_{\xi}).
	$
	By Proposition \ref{prop: ecc} and Proposition \ref{prop: Shintani}, the class $\partial^k\cG(\xi)$ maps 
	through\eqref{eq: isom main} to 
	\[
		\sum_{\sigma\in\Delta_\xi\backslash\Phi_g} \partial^k\cG_\sigma(\xi)
		=\sum_{\sigma\in\Delta_\xi\backslash\Phi_g} \zeta_{\sigma}(\xi,(-k,\ldots,-k)).
	\]
	Our assertion now follows from \eqref{eq: Shintani}.
\end{proof}

\begin{corollary}
Assume that the narrow class number of $F$ is \textit{one},
	and let $\chi\colon\Cl^+_F(\frf)\rightarrow\bbC^\times$ be a finite primitive Hecke character of $F$ of 
	conductor $\frf\neq(1)$.
	If we let $U[\frf]\coloneqq\bbT[\frf]\setminus\{1\}$, then we have
	\[
		L(\chi, -k)=\sum_{\xi\in U[\frf]/\Delta}c_\chi(\xi)\partial^k\cG(\xi)
	\]
	for any integer $k\geq0$.
\end{corollary}

\begin{proof}
	The result follows from Theorem \ref{theorem: main} and Proposition \ref{prop: Hecke}.
\end{proof}

The significance of this result is that
the special values at negative integers of \textit{any} finite Hecke character of $F$ may be expressed as a linear combination of special values of the derivatives of a single canonical cohomology class, the Shintani class $\cG$
in $H^{g-1}(U/\Delta,\sO_\bbT)$.

\subsection*{Acknowledgement} 
The authors would like to thank the KiPAS program FY2014--2018 of the Faculty of Science and Technology at 
Keio University, especially Professors Yuji Nagasaka and Masato Kurihara, for providing an excellent environment making this research possible.  
The authors thank Shinichi Kobayashi for discussion.
The authors would also like to thank Yoshinosuke Hirakawa, Yoshinori Kanamura, Hideki Matsumura and Takuki Tomita
for extensive discussion concerning Shintani's work.

\end{document}